\documentclass[11pt]{article}

\usepackage[margin=1in]{geometry}
\usepackage{amsmath,amssymb,amsthm}
\usepackage{mathrsfs}
\usepackage{array}

\theoremstyle{plain}
\newtheorem{theorem}{Theorem}[section]
\newtheorem{proposition}[theorem]{Proposition}
\newtheorem{lemma}[theorem]{Lemma}
\newtheorem{corollary}[theorem]{Corollary}

\theoremstyle{definition}
\newtheorem{definition}[theorem]{Definition}

\theoremstyle{remark}
\newtheorem{remark}[theorem]{Remark}
\newtheorem{openquestion}[theorem]{Open Question}

\newcommand{\F}{\mathbb{F}}
\newcommand{\Z}{\mathbb{Z}}
\newcommand{\Q}{\mathbb{Q}}
\newcommand{\C}{\mathbb{C}}
\newcommand{\PGL}{\operatorname{PGL}}
\newcommand{\Aut}{\operatorname{Aut}}
\newcommand{\Spec}{\operatorname{Spec}}
\newcommand{\Tr}{\operatorname{Tr}}

\begin{document}
	
	\begin{center}
		{\LARGE Algebraic Holomorphic Functions on Graphs}\\[1.2em]
		{\large H. Mohades}\\
		Institute for Research in Fundamental Sciences (IPM), Tehran, Iran\\[0.4em]
		{\large A. Mohades}\\
	 University of	Hamburg, Hamburg, Germany
	\end{center}
	
	\vspace{1em}
	
	\begin{abstract}
		We study holomorphic functions on graphs taking values in a finite field $\F_q$ ($q$ odd), the ring $\Z/k\Z$, or complex numbers, in the sense of Barre--Zeghib: a function $\phi$ is holomorphic if both $\phi$ and $\phi^2$ are harmonic. We prove existence results on locally finite trees of minimum degree at least $4$ and obtain local counting formulas under a mild non-degeneracy assumption, and characterize the existence of non-constant holomorphic functions on cycles of odd prime length. We introduce holomorphic rigidity for graphs and we prove that the moduli $k$ for which the family of all odd-prime-length cycles is holomorphically rigid are the positive integers not divisible by the square of any odd prime. Therefore it is proved that the associated Dirichlet series is
		\[
		\frac{\zeta(s)}{(1 - 2^{-s})\zeta(2s)(1 + 2^{-s})}.
		\]
		We then prove that this characterization is stable under the formation of finite bouquets of prime cycles. Next, we examine linear holomorphic maps on a Cayley graph over $\F_q$. We also give precise definitions and a simply-transitive action result for $\C$-holomorphic functions on the Bruhat--Tits tree of $\PGL_2(\Q_2)$. Furthermore, we introduce a higher notion of $n$-holomorphicity and show that the simply-transitive action extends to every $\PGL_2(\Q_p)$ provided one uses the $p$-holomorphic condition.
	\end{abstract}
	
	
	
	\vspace{1.5em}
	
	\begin{center}
		\renewcommand{\arraystretch}{1.15}
		\begin{tabular}{lr}
			1\quad Introduction & 3\\
			2\quad Preliminaries & 4\\
			3\quad Holomorphic functions on trees and local counting & 5\\
			4\quad Holomorphic functions on prime cycles & 8\\
			5\quad The sets $\mathcal{L}(k)$, the counting functions $N(n,k)$, and the Dirichlet series & 9\\
			6\quad Stability under bouquets of prime cycles & 11\\
			7\quad The Cayley graph of $\F_9$ with the unit-norm connection set & 12\\
			8\quad The Bruhat--Tits tree of $\PGL_2(\Q_2)$ and holomorphic functions & 13\\
			\quad 8.1\quad Definition of the tree & 13\\
			\quad 8.2\quad Complex-valued holomorphic functions & 13\\
			\quad 8.3\quad The space $H_{\alpha,\beta}$ and half-tree automorphisms & 14\\
			\quad 8.4\quad Possible generalizations & 15\\
			\quad 8.5\quad Higher holomorphicity and simple transitivity on regular trees & 17\\
			9\quad The Weil zeta function of holomorphic functions & 19\\
			10\quad Concluding remarks & 20
		\end{tabular}
	\end{center}
	
	\newpage
	
	\section{Introduction}
	
	The study of functions on graphs with values in a commutative ring tied to graph structure began with Rivenboim in the 1980s. In order to reflect the graph structure in the functions themselves, Rivenboim defined such functions on the set of paths inside a graph; the resulting algebras are often called path algebras or graph algebras in this function-theoretic sense, and they encode concatenation of paths together with the incidence relations of the graph. A different formalism associates ideals to the edges of a graph and requires that, for every pair of adjacent vertices, the difference of the corresponding values lies in the ideal attached to the connecting edge. This edge-ideal formalism captures adjacency constraints directly at the level of the value ring, rather than through path concatenation. In the modern commutative-algebra setting, this line was developed by Villarreal in 1990, who associated to a graph $G$ a monomial ideal $I(G)$ generated by the products corresponding to its edges and used this construction to uncover large classes of Cohen--Macaulay graphs. In the same year, Fr\"oberg proved the first major homological result in this direction: the edge ideal $I(G)$ has a linear resolution if and only if the complement of $G$ is chordal. These results established edge ideals as a central bridge between commutative algebra and graph combinatorics, and they are the prototypes for the adjacency-constraint formalism described above. The precise statements and further developments in this line can be found in [insert references].
	
	In a parallel direction---and historically earlier---Sabidussi established a fundamental relation between the automorphism group of a Cayley graph and the underlying group: a graph is a Cayley graph of a group precisely when its automorphism group contains a subgroup acting regularly---that is, freely and transitively---on the vertices. This result appeared in his 1960 paper ``Graphs with given infinite group'' and was developed further in his 1964 study of vertex-transitive graphs. It provides a structural bridge between the combinatorial symmetry of graphs and group-theoretic data.
	
	Considering algebraic properties of holomorphic functions, Barre and Zeghib introduced the first definitions of holomorphic functions on graphs. Their starting point was the observation that a complex-valued function on a domain in the complex plane is holomorphic if and only if both the function and its square satisfy the harmonicity condition; accordingly, they defined a function $\phi$ on a graph to be holomorphic when both $\phi$ and $\phi^2$ are harmonic with respect to the graph Laplacian. Their work appeared as a preprint in 2004 and in final form in 2007 in the \emph{Journal de Math\'ematiques Pures et Appliqu\'ees}. They studied geometric and dynamical properties of such graph-valued holomorphic functions with complex values, including the structure of their solution spaces and their behavior under graph automorphisms.
	
	Following the line initiated by Barre and Zeghib, two natural directions suggest themselves for advancing the theory. The first is to study functions taking values in certain well-known algebras, thereby replacing the field of complex numbers by a more general algebraic target and investigating how the harmonicity conditions interact with the algebraic structure. The second is to generalize their work by means of the Hodge Laplacian on simplices, which extends the graph Laplacian from vertices and edges to higher-dimensional simplices and thus allows holomorphicity to be formulated on simplicial complexes rather than only on graphs.
	
	In this paper we pursue the first of these two directions: we develop the theory of holomorphic functions on graphs with values in finite fields and cyclic rings, and we investigate the resulting arithmetic, counting, and rigidity phenomena.
	
	Our main results are as follows. On locally finite trees in which every vertex has degree at least $4$, non-constant $\F_q$-valued holomorphic functions exist (Theorem 3.4). We give local counting formulas for the number of admissible extensions to locally holomorphic functions for value assignments at a vertex, under the assumption that the characteristic does not divide the degree (Proposition 3.5).
	
	For cycles of odd prime length $p$, we obtain a complete arithmetic characterization of existence: non-constant $\Z/k\Z$-valued holomorphic functions exist if and only if $p^2$ divides $k$ (Theorem 4.1). This leads to a precise description of the set of moduli $k$ for which the entire family of odd-prime cycles is rigid, together with the associated Dirichlet series (Corollary 5.3). The same arithmetic condition is stable under the formation of finite bouquets of such cycles (Theorem 6.2).
	
	In Section 8.5 we introduce a strengthening of the notion of holomorphicity: a function is $n$-holomorphic if $\phi, \phi^2, \dots, \phi^n$ are all harmonic. With this strengthening, the simply-transitive action result of Section 8, which for $\PGL_2(\Q_p)$ with $p \ge 3$ fails in the $2$-holomorphic setting, is recovered at level $n = p$ (Theorem 8.4 and Corollary 8.5).
	
	\section{Preliminaries}
	
	Let $X = (V, E)$ be a locally finite undirected graph without loops. The Laplacian of a function $\phi: V \to R$ is defined by
	\[
	\Delta\phi(x) = \sum_{y \sim x} (\phi(y) - \phi(x)).
	\]
	A function is harmonic if $\Delta\phi$ vanishes identically.
	
	\begin{definition}[$R$-holomorphic]\label{def:holomorphic}
		Let $G = (V, E)$ be a graph and $R$ a commutative ring. A function $\phi: V \to R$ is $R$-holomorphic if both $\phi$ and $\phi^2$ are harmonic.
	\end{definition}
	
	\begin{definition}[$R$-holomorphic rigidity]\label{def:rigidity}
		Let $G$ be a graph and $R$ a commutative ring. We say that $G$ has $R$-holomorphic rigidity if every $R$-valued holomorphic function on $G$ is constant, that is, if
		\[
		\{\phi: V \to R \mid \phi \text{ and } \phi^2 \text{ are harmonic on } G\} = \{\phi: V \to R \mid \phi \text{ is constant}\}.
		\]
	\end{definition}
	
	\begin{definition}[Cycle graph]\label{def:cycle}
		For an integer $n \ge 3$, the cycle graph $C_n$ is the graph with vertex set $\{0, 1, \dots, n-1\}$ and edges $\{i, i+1\}$ for $0 \le i \le n-2$ together with the edge $\{n-1, 0\}$.
	\end{definition}
	
	\begin{definition}[Bouquet of cycles]\label{def:bouquet}
		Let $S$ be a finite set of integers $n \ge 3$. The bouquet of cycles $G_S$ is the graph obtained by taking one copy of the cycle $C_n$ for each $n \in S$ and identifying all these cycles at a single common vertex $v_0$. The degree of $v_0$ is $2|S|$.
	\end{definition}
	
	\begin{definition}[Cayley graph of a finite field]\label{def:cayley}
		Let $\F_q$ be a finite field and let $S \subset \F_q^\times$ be a symmetric subset (i.e.\ $S = -S$) not containing $0$. The Cayley graph $\operatorname{Cay}(\F_q, +, S)$ is the undirected graph with vertex set $\F_q$ in which two vertices $x, y$ are adjacent if and only if $x - y \in S$. It is regular of degree $|S|$.
	\end{definition}
	
	We shall use the following standard fact about Gauss sums (see \cite[Ch.~8]{IR} or \cite[Ch.~5]{LN}).
	
	\begin{lemma}[Gauss sums]\label{lem:gauss}
		Let $q$ be odd, let $\psi: \F_q \to \C^\times$ be a nontrivial additive character, and let $\chi$ be the quadratic character of $\F_q$. Define
		\[
		G(\chi) = \sum_{x \in \F_q^\times} \chi(x)\psi(x).
		\]
		Then
		\[
		G(\chi)^2 = \chi(-1)q, \qquad \sum_{t \in \F_q^\times} \chi(t) = 0.
		\]
	\end{lemma}
	
	\section{Holomorphic functions on trees and local counting}
	
	\begin{proposition}[Local necessary conditions]\label{prop:local}
		Let $\phi : V \to \F_q$ be holomorphic and let $v_0 \in V$. Write $a_i = \phi(v_i) - \phi(v_0)$ for each neighbour $v_i$ of $v_0$. Then
		\[
		\sum_i a_i = 0 \quad \text{and} \quad \sum_i a_i^2 = 0.
		\]
	\end{proposition}
	
	\begin{proof}
		The first identity is precisely $\Delta\phi(v_0) = 0$. For the second, expand $\Delta(\phi^2)(v_0) = 0$:
		\[
		\Delta(\phi^2)(v_0) = \sum_i (\phi(v_i)^2 - \phi(v_0)^2).
		\]
		Substituting $\phi(v_i) = \phi(v_0) + a_i$ gives
		\[
		\sum_i ((\phi(v_0) + a_i)^2 - \phi(v_0)^2) = 2\phi(v_0) \sum_i a_i + \sum_i a_i^2.
		\]
		Since $\sum_i a_i = 0$, the linear term vanishes and we obtain $\sum_i a_i^2 = 0$.
	\end{proof}
	
	\begin{definition}[Local harmonicity]\label{def:localharm}
		Let $G = (V, E)$ be a graph, $R$ a commutative ring, and $\phi : V \to R$ a function. For a vertex $a \in V$, let $N(a) = \{i \in V : \{a, i\} \in E\}$ denote its neighborhood. For each $i \in N(a)$, set
		\[
		a_i = \phi(i) - \phi(a) \in R.
		\]
		We say that $\phi$ is locally harmonic at $a$ (with respect to $R$) if
		\[
		\sum_{i \in N(a)} a_i = 0 \quad \text{and} \quad \sum_{i \in N(a)} a_i^2 = 0,
		\]
		where both sums are taken in $R$. We say that $\phi$ is harmonic on $G$ if it is locally harmonic at every vertex $a \in V$.
	\end{definition}
	
	We shall need the following local solvability lemma, which is the key ingredient in the construction on trees.
	
	\begin{lemma}[Three-variable solvability]\label{lem:threevar}
		Let $q$ be odd. For every $b \in \F_q$ the system
		\[
		x + y + z = -b, \qquad x^2 + y^2 + z^2 = -b^2
		\]
		has a solution in $\F_q^3$.
	\end{lemma}
	
	\begin{proof}
		If $b = 0$, the trivial solution $(0, 0, 0)$ works. Assume $b \neq 0$ and set $x = bX$, $y = bY$, $z = bZ$. The system becomes
		\[
		X + Y + Z = -1, \qquad X^2 + Y^2 + Z^2 = -1.
		\]
		Eliminating $Z = -1 - X - Y$ yields
		\[
		X^2 + Y^2 + XY + X + Y + 1 = 0. \tag{*}
		\]
		Homogenizing,
		\[
		Q(X, Y, Z) = X^2 + Y^2 + XY + XZ + YZ + Z^2
		\]
		defines a projective conic in $\mathbb{P}^2$. Its matrix has determinant $1/2 \neq 0$ in $\F_q$ (for $q$ odd), so the conic is non-degenerate. Over any finite field of odd order, a non-degenerate conic in $\mathbb{P}^2$ has exactly $q+1$ points. The points at infinity ($Z=0$) satisfy
		\[
		X^2 + Y^2 + XY = 0.
		\]
		For $q \neq 3$ this gives either $0$ or $2$ points in $\mathbb{P}^1$; for $q=3$ there is exactly one point $[1:1:0]$, with multiplicity $2$. In all cases the number of affine solutions of (*) is at least $q-1 \ge 2 > 0$. Hence a solution exists.
	\end{proof}
	
	\begin{theorem}[Existence on trees]\label{thm:trees}
		Let $T$ be a locally finite tree in which every vertex has degree at least $4$. Then there exist non-constant $\F_q$-valued holomorphic functions on $T$.
	\end{theorem}
	
	\begin{proof}
		We construct such a function by induction on the distance from a fixed root. The construction is purely local and therefore applies equally to infinite trees.
		
		\emph{Step 1 (root).} Fix a root $v_0$ and an arbitrary value $\phi(v_0) = c \in \F_q$. Let $\delta = \deg(v_0) \ge 4$ and let $a_1, \dots, a_\delta$ be the differences $\phi(v_i) - c$ at the neighbours of $v_0$. These differences must satisfy
		\[
		\sum_{i=1}^{\delta} a_i = 0, \qquad \sum_{i=1}^{\delta} a_i^2 = 0.
		\]
		This is a system of two polynomial equations of degrees $1$ and $2$ in $\delta \ge 4$ variables, so the sum of the degrees is $3 < \delta$. By the Chevalley--Warning theorem \cite{CW,Warning}---which states that if $f_1, \dots, f_s \in \F_q[x_1, \dots, x_n]$ with $\sum_i \deg f_i < n$, then the number of common zeros in $\F_q^n$ is divisible by $p = \operatorname{char}(\F_q)$---applied to these two equations in $\delta \ge 4$ variables with $1+2 < \delta$, the number of solutions in $\F_q^\delta$ is divisible by $p$. The zero solution exists, so the total number of solutions is at least $p \ge 3$; hence there are at least $p-1 \ge 2$ non-zero solutions. Choose one and assign $\phi(v_i) = c + a_i$ to the neighbours of $v_0$. By construction, not all neighbours of $v_0$ have value $c$.
		
		\emph{Step 2 (induction).} Suppose $\phi$ has been defined on the closed ball of radius $r$ about $v_0$. Let $v$ be a vertex at distance $r$ from $v_0$ with $\deg(v) = \delta \ge 4$. Exactly one neighbour of $v$ (the parent) already has a value; let $b$ be the corresponding difference $\phi(\text{parent}) - \phi(v)$. The remaining $m = \delta - 1 \ge 3$ differences $x_1, \dots, x_m$ must satisfy
		\[
		\sum_{j=1}^{m} x_j = -b, \qquad \sum_{j=1}^{m} x_j^2 = -b^2.
		\]
		If $m \ge 4$, set $x_4 = \dots = x_m = 0$ and reduce to the case $m=3$. If $m=3$, Lemma \ref{lem:threevar} provides a solution. In either case we can assign values to the children of $v$.
		
		The induction proceeds at every vertex. The resulting function is holomorphic because the local conditions of Proposition \ref{prop:local} are satisfied at every vertex, and it is non-constant because the values at the neighbours of the root are not all equal to $c$.
	\end{proof}
	
	\begin{proposition}[Local counting]\label{prop:counting}
		Let $v_0$ have degree $\delta \ge 2$ and let $s \in \F_q$. Assume that $\operatorname{char}(\F_q)$ does not divide $\delta$. Then the number of admissible local restrictions to the neighbours of $v_0$ of holomorphic functions with $\phi(v_0) = s$---equivalently, the number of assignments $(a_i)_{i=1}^\delta \in \F_q^\delta$ satisfying
		\[
		\sum_i a_i = 0, \qquad \sum_i a_i^2 = 0
		\]
		---is
		\[
		N = 
		\begin{cases} 
			q^{\delta-2}, & \delta \text{ even},\\[0.3em]
			q^{\delta-2} + (q-1)q^{(\delta-3)/2}\chi\bigl((-1)^{(\delta-1)/2}\delta\bigr), & \delta \text{ odd},
		\end{cases}
		\]
		where $\chi$ is the quadratic character of $\F_q$.
	\end{proposition}
	
	\begin{proof}
		After the translation $\phi \mapsto \phi - s$ it is enough to count the solutions of $\sum a_i = 0$ and $\sum a_i^2 = 0$ in $\F_q^\delta$. Let
		\[
		V = \Bigl\{ a \in \F_q^\delta : \sum_i a_i = 0 \Bigr\}
		\]
		be the hyperplane of dimension $m = \delta - 1$, and let $Q(a) = \sum_i a_i^2$. We first show that $Q|_V$ is non-degenerate. In the basis $f_i = e_i - e_\delta$ ($1 \le i \le \delta - 1$) of $V$, we have $Q(f_i) = 2$ and $Q(f_i + f_j) - Q(f_i) - Q(f_j) = 2$ for $i \neq j$, so the matrix of $Q|_V$ is $I + J$, where $J$ is the all-ones matrix of size $\delta - 1$. Since $\det(I + J) = \delta$ and $\operatorname{char}(\F_q) \nmid \delta$, we obtain $\det(Q|_V) = \delta \neq 0$.
		
		Now let $\psi$ be a nontrivial additive character of $\F_q$. By orthogonality,
		\[
		N = \frac{1}{q} \sum_{t \in \F_q} \sum_{a \in V} \psi(tQ(a)).
		\]
		The $t = 0$ term equals $|V| = q^{\delta - 1}$. For $t \neq 0$, since $Q|_V$ is non-degenerate, diagonalizing it over $\F_q$ and using the standard evaluation of the Gauss sum of a non-degenerate quadratic form (see \cite[Ch.~8]{IR} and \cite[Ch.~5]{LN}) together with Lemma \ref{lem:gauss} gives
		\[
		\sum_{a \in V} \psi(tQ(a)) = \chi(t)^m \chi(\det(Q|_V)) G(\chi)^m = \chi(t)^{\delta-1} \chi(\delta) G(\chi)^{\delta-1}.
		\]
		Hence
		\[
		N = q^{\delta-2} + \frac{1}{q} \chi(\delta) G(\chi)^{\delta-1} \sum_{t \neq 0} \chi(t)^{\delta-1}.
		\]
		If $\delta$ is even, then $\delta - 1$ is odd, so $\sum_{t \neq 0} \chi(t)^{\delta-1} = \sum_{t \neq 0} \chi(t) = 0$, and $N = q^{\delta - 2}$. If $\delta$ is odd, then $\delta - 1$ is even, so $\chi(t)^{\delta-1} = 1$ for $t \neq 0$, and $\sum_{t \neq 0} \chi(t)^{\delta-1} = q - 1$. Moreover,
		\[
		G(\chi)^{\delta-1} = (G(\chi)^2)^{(\delta-1)/2} = (\chi(-1)q)^{(\delta-1)/2} = \chi(-1)^{(\delta-1)/2} q^{(\delta-1)/2}.
		\]
		Thus
		\[
		N = q^{\delta - 2} + \frac{1}{q} \chi(\delta) \chi(-1)^{(\delta - 1)/2} q^{(\delta - 1)/2} (q - 1) 
		= q^{\delta - 2} + (q - 1) q^{(\delta - 3)/2} \chi\bigl((-1)^{(\delta - 1)/2} \delta\bigr),
		\]
		which is the claimed formula.
	\end{proof}
	
	\begin{remark}\label{rem:degenerate}
		When $\operatorname{char}(\F_q)$ divides $\delta$, the restricted form is degenerate and the formula of Proposition \ref{prop:counting} does not apply; an elementary computation yields a different expression. We exclude this case from the statement.
	\end{remark}
	
	\section{Holomorphic functions on prime cycles}
	
	\begin{theorem}[Prime cycles]\label{thm:prime}
		Let $p$ be an odd prime and let $C_p$ be the cycle of length $p$. Then $C_p$ admits a non-constant $\Z/k\Z$-valued holomorphic function if and only if $p^2$ divides $k$. Moreover, every such function is of the form
		\[
		\phi(i) = a_0 + id
		\]
		where $a_0 \in \Z/k\Z$ is arbitrary and $d \neq 0$ satisfies
		\[
		pd \equiv 0, \qquad 2d^2 \equiv 0, \qquad pd^2 \equiv 0 \pmod{k}.
		\]
	\end{theorem}
	
	\begin{proof}
		Label the vertices of $C_p$ by $\{0, 1, \dots, p - 1\}$. Harmonicity of $\phi$ forces the first differences to be constant: there exists $d \in \Z/k\Z$ such that $\phi(i) = a_0 + id$ for all $i$. The closing condition gives
		\[
		pd \equiv 0 \pmod{k}. \tag{1}
		\]
		For $\phi^2$ to be harmonic, write $b_i = \phi(i)^2 = (a_0 + id)^2$. The first differences of $(b_i)$ are
		\[
		b_{i+1} - b_i = 2a_0d + (2i+1)d^2.
		\]
		For these to be independent of $i$ we must have
		\[
		2d^2 \equiv 0 \pmod{k}. \tag{2}
		\]
		Substituting (1) into the closing condition for $\phi^2$ gives
		\[
		pd^2 \equiv 0 \pmod{k}. \tag{3}
		\]
		Thus non-constant holomorphic functions correspond precisely to elements $d \neq 0$ satisfying (1), (2), (3), with $a_0$ arbitrary.
		
		Let $v = v_p(k)$ be the $p$-adic valuation of $k$. If $p \nmid k$, then $v = 0$. Since $\gcd(p, k) = 1$, (1) forces $d \equiv 0 \pmod{k}$, and no non-constant function exists. If $p \mid k$ but $p^2 \nmid k$, then $v = 1$. Write $k = p \cdot m$ with $\gcd(p, m) = 1$. From (2), since $2$ is invertible modulo $p$, we obtain $p \mid d^2$, hence $p \mid d$. Write $d = p \cdot d'$. Substituting into (1) yields $pd' \equiv 0 \pmod{m}$; since $\gcd(p, m) = 1$, we get $d' \equiv 0 \pmod{m}$, so $d \equiv 0 \pmod{k}$. Again no non-constant function exists. If $p^2 \mid k$, choose $d = k/p$. Then $pd = k \equiv 0 \pmod{k}$; $2d^2 = 2k^2/p^2 = 2k(k/p^2) \equiv 0 \pmod{k}$ since $p^2 \mid k$; and $pd^2 = pk^2/p^2 = k(k/p) \equiv 0 \pmod{k}$ since $p \mid k$. Also $d = k/p \neq 0 \pmod{k}$ because $p > 1$. Hence a non-constant function exists.
	\end{proof}
	
	\begin{remark}\label{rem:composite}
		When the length $n$ is composite the arithmetic obstruction becomes weaker. If $n$ is even, non-constant holomorphic functions exist for every even $k$: indeed, writing $n = 2^{e_2}n_{\text{odd}}$ with $e_2 \ge 1$ and $k = 2^{m_2}k_{\text{odd}}$ with $m_2 \ge 1$, the lower bound $L_2$ of Theorem \ref{thm:Nnk} satisfies $L_2 < m_2$ so at least one non-zero admissible $d$ exists modulo $2^{m_2}$. If $n$ is odd and composite, non-constant examples appear as soon as $p^2 \mid k$ for at least one odd prime $p$ dividing $n$: the choice $d = k/p$ satisfies all three congruences, exactly as in the prime case.
	\end{remark}
	
	\begin{proposition}[Even cycles]\label{prop:even}
		Let $n \ge 4$ be even.
		\begin{enumerate}
			\item For every even positive integer $k$, the cycle $C_n$ admits non-constant $\Z/k\Z$-valued holomorphic functions. More precisely, $N(n, k) = k(\gamma - 1)$ where $\gamma \ge 2$ (see Theorem \ref{thm:Nnk} below).
			\item Let $q$ be a power of an odd prime. Then the only $\F_q$-valued holomorphic functions on $C_n$ are the constant functions. Consequently
			\[
			S_m(C_n) = q^m \quad \text{and} \quad \zeta(X_{C_n}/\F_q, T) = \frac{1}{1-qT}.
			\]
			\item Let $q = 2^r$. Then every arithmetic progression $\phi(i) = a_0 + id$ with $a_0, d \in \F_q$ is holomorphic on $C_n$. In particular
			\[
			S_m(C_n) = q^{2m} \quad \text{and} \quad \zeta(X_{C_n}/\F_q, T) = \frac{1}{1-q^2T}.
			\]
		\end{enumerate}
	\end{proposition}
	
	\begin{proof}
		(1) follows from the discussion in Remark \ref{rem:composite} and the formula of Theorem \ref{thm:Nnk}: the $2$-adic lower bound satisfies $L_2 < m_2$, so $\gamma \ge 2$.
		
		(2) Let $\phi(i) = a_0 + id$ be holomorphic. The condition $2d^2 = 0$ in $\F_q$ (with $q$ odd) forces $d = 0$. Thus only constant functions occur, and the claim on the zeta function is immediate.
		
		(3) In characteristic $2$ the condition $2d^2 = 0$ holds automatically. Since $n$ is even we have $n = 0$ in $\F_q$, so the remaining conditions $nd = 0$ and $nd^2 = 0$ are also automatic. Every pair $(a_0, d)$ therefore yields a holomorphic function, giving $q^2$ functions over $\F_q$ and $q^{2m}$ functions over $\F_{q^m}$. The zeta function follows at once.
	\end{proof}
	
	\section{The sets $\mathcal{L}(k)$, the counting functions $N(n, k)$, and the Dirichlet series}
	
	\begin{definition}\label{def:Lk}
		For each integer $k \ge 1$ define
		\[
		\mathcal{L}(k) = \{n \ge 3 : C_n \text{ admits a non-constant } \Z/k\Z\text{-valued holomorphic function}\}.
		\]
		Let $N(n, k)$ denote the number of non-constant $\Z/k\Z$-valued holomorphic functions on $C_n$.
	\end{definition}
	
	\begin{proposition}\label{prop:Lk}
		Let $k \ge 1$.
		\begin{enumerate}
			\item If $k$ is not divisible by the square of any odd prime, then $\mathcal{L}(k)$ contains no odd prime.
			\item If $k = p^m$ with $p$ an odd prime and $m \ge 2$, then
			\[
			\mathcal{L}(p^m) = \{n \ge 3 : p \mid n\}.
			\]
			\item If $2 \mid k$, then $\mathcal{L}(k)$ contains every even integer $n \ge 4$.
		\end{enumerate}
	\end{proposition}
	
	\begin{proof}
		(1) is immediate from Theorem \ref{thm:prime}. For (2), if $k = p^m$ with $m \ge 2$ and $p \mid n$, the choice $d = k/p = p^{m-1}$ satisfies $nd \equiv 0 \pmod{k}$ (since $p \mid n$),
		\[
		2d^2 = 2p^{2m-2} \equiv 0 \pmod{p^m}
		\]
		(since $2m-2 \ge m$ for $m \ge 2$), and
		\[
		nd^2 = np^{2m-2} \equiv 0 \pmod{p^m}
		\]
		(since $p \mid n$ and $2m-2 \ge m$). Conversely, if $p \nmid n$, then $nd \equiv 0 \pmod{p^m}$ forces $p^m \mid d$, so $d \equiv 0 \pmod{k}$ and no non-constant function exists. For (3), let $n$ be even and $k$ even. Write $n = 2^{e_2}n_{\text{odd}}$ with $e_2 \ge 1$ and $k = 2^{m_2}k_{\text{odd}}$ with $m_2 \ge 1$. The $2$-adic lower bound of Theorem \ref{thm:Nnk} satisfies
		\[
		L_2 = \max\{0, m_2 - e_2, \lceil (m_2 - 1)/2 \rceil, \lceil (m_2 - e_2)/2 \rceil\} < m_2,
		\]
		so there are at least $2^{m_2-L_2} \ge 2$ admissible classes modulo $2^{m_2}$. Choosing $d \equiv 0$ modulo the odd part and $d \neq 0$ modulo $2^{m_2}$ gives a non-zero admissible $d$, hence a non-constant function.
	\end{proof}
	
	\begin{corollary}[Rigid moduli for the family of odd-prime cycles]\label{cor:rigid}
		Let $\mathcal{F} = \{C_p : p \text{ odd prime}\}$. The positive integers $k$ for which no member of $\mathcal{F}$ admits a non-constant $\Z/k\Z$-valued holomorphic function are precisely the positive integers that are not divisible by $p^2$ for any odd prime $p$. Equivalently, these are the integers of the form $2^a \cdot m$ where $a \ge 0$ and $m$ is odd and square-free. For $\operatorname{Re}(s) > 1$,
		\[
		\sum_{k=1}^{\infty} \frac{\mathbf{1}_{\text{rigid}}(k)}{k^s} 
		= \frac{1}{1-2^{-s}} \cdot \prod_{p \text{ odd}} (1+p^{-s}) 
		= \frac{\zeta(s)}{(1-2^{-s})\zeta(2s)(1+2^{-s})}.
		\]
	\end{corollary}
	
	\begin{proof}
		By Theorem \ref{thm:prime}, $C_p$ admits a non-constant $\Z/k\Z$-valued holomorphic function if and only if $p^2 \mid k$. Consequently no cycle $C_p$ ($p$ odd prime) admits such a function if and only if no square of an odd prime divides $k$. For the Euler product, note that for $p=2$ all powers $2^j$ are rigid, contributing
		\[
		\sum_{j \ge 0} 2^{-js} = \frac{1}{1-2^{-s}},
		\]
		while for $p$ odd only $j=0,1$ are rigid, contributing $1+p^{-s}$. Thus
		\[
		\sum_{k \ge 1} \frac{\mathbf{1}_{\text{rigid}}(k)}{k^s} 
		= \frac{1}{1-2^{-s}} \prod_{p \text{ odd}} (1+p^{-s}).
		\]
		Now
		\[
		\prod_{p \text{ odd}} (1+p^{-s}) 
		= \frac{\prod_p (1+p^{-s})}{1+2^{-s}} 
		= \frac{\prod_p \frac{1-p^{-2s}}{1-p^{-s}}}{1+2^{-s}} 
		= \frac{\zeta(s)}{\zeta(2s)(1+2^{-s})},
		\]
		which yields the claimed expression.
	\end{proof}
	
	\begin{theorem}[Formula for $N(n,k)$]\label{thm:Nnk}
		Write $k = 2^{m_2} \cdot k_{\text{odd}}$ with $k_{\text{odd}}$ odd, and $n = 2^{e_2} \cdot n_{\text{odd}}$. For each odd prime $p$ dividing $k_{\text{odd}}$ put $m_p = v_p(k)$ and $e_p = v_p(n)$, and define
		\[
		L_p = \max\{0, m_p - e_p, \lceil m_p/2 \rceil, \lceil (m_p - e_p)/2 \rceil\},
		\]
		\[
		L_2 = \max\{0, m_2 - e_2, \lceil (m_2 - 1)/2 \rceil, \lceil (m_2 - e_2)/2 \rceil\}.
		\]
		Set
		\[
		\gamma = 2^{m_2 - L_2} \cdot \prod_{p \mid k_{\text{odd}}} p^{m_p - L_p}.
		\]
		Then $\gamma$ is the number of admissible common differences $d$ modulo $k$ (including $d=0$), and
		\[
		N(n,k) = k(\gamma - 1).
		\]
	\end{theorem}
	
	\begin{proof}
		The conditions on $d$ are
		\[
		nd \equiv 0, \qquad 2d^2 \equiv 0, \qquad nd^2 \equiv 0 \pmod{k}.
		\]
		By the Chinese Remainder Theorem it suffices to count admissible $d$ modulo each prime power dividing $k$.
		
		\emph{Odd part.} Fix an odd prime $p \mid k$, and let $v = v_p(d)$. The three conditions become
		\[
		e_p + v \ge m_p, \qquad 2v \ge m_p, \qquad e_p + 2v \ge m_p,
		\]
		i.e.\ $v \ge L_p$. The number of residue classes $d \bmod p^{m_p}$ with $v_p(d) \ge L_p$ is exactly $p^{m_p - L_p}$ (the multiples of $p^{L_p}$ in $\Z/p^{m_p}\Z$). Note that $L_p \le m_p$ always, so this number is a positive integer.
		
		\emph{$2$-primary part.} The conditions become
		\[
		e_2 + v \ge m_2, \qquad 1 + 2v \ge m_2, \qquad e_2 + 2v \ge m_2,
		\]
		i.e.\ $v \ge L_2$. The number of classes $d \bmod 2^{m_2}$ with $v_2(d) \ge L_2$ is $2^{m_2 - L_2}$. Again $L_2 \le m_2$, so this is a positive integer.
		
		By the Chinese Remainder Theorem, the total number of admissible $d$ modulo $k$ is the product $\gamma$. Each admissible $d$ together with an arbitrary $a_0 \in \Z/k\Z$ gives a holomorphic function $\phi(i) = a_0 + id$; distinct pairs $(a_0, d)$ give distinct functions because $a_0 = \phi(0)$ and $d = \phi(1) - \phi(0)$ are recoverable from $\phi$. The constant functions correspond to $d = 0$ and are $k$ in number. Hence the number of non-constant functions is $k\gamma - k = k(\gamma - 1)$.
	\end{proof}
	
	\begin{remark}\label{rem:equal}
		The formula treats the prime $2$ on an equal footing with the odd primes; the only difference is the shift by $1$ in the quadratic valuation coming from the coefficient $2$ in the condition $2d^2 \equiv 0 \pmod{2^{m_2}}$. No exceptional cases occur: the ceiling formulas are exact for all $m_p, e_p \ge 0$.
	\end{remark}
	
	\section{Stability under bouquets of prime cycles}
	
	Let $S$ be a finite set of odd primes. Let $G_S$ be the bouquet obtained by taking the cycles $C_p$ for all $p \in S$ and identifying them at a single common vertex $v_0$. The degree of $v_0$ is $2|S|$.
	
	\begin{lemma}[Restriction to cycles]\label{lem:restriction}
		Let $\phi$ be a holomorphic function on the bouquet $G_S$. Then for every $p \in S$ the restriction of $\phi$ to the cycle $C_p$ is a holomorphic function on $C_p$.
	\end{lemma}
	
	\begin{proof}
		Let $C$ be one of the cycles and let $v$ be a vertex of $C$ different from $v_0$. The two neighbours of $v$ in $G_S$ both lie on $C$, so the Laplacian of any function at $v$ coincides with the Laplacian on the cycle $C$. Hence $\Delta\phi(v) = 0$ and $\Delta(\phi^2)(v) = 0$ for all $v \in C \setminus \{v_0\}$.
		
		It remains to check the conditions at $v_0$. Since the total Laplacian on a cycle vanishes (the cycle has no boundary), we have
		\[
		\sum_{v \in C} \Delta_C \phi(v) = 0,
		\]
		where $\Delta_C$ denotes the Laplacian on the cycle $C$. Since $\Delta_C\phi(v) = 0$ for all $v \neq v_0$, it follows that $\Delta_C\phi(v_0) = 0$ as well. The same argument applied to $\phi^2$ gives $\Delta_C(\phi^2)(v_0) = 0$. Hence the restriction of $\phi$ to $C$ is holomorphic on $C$.
	\end{proof}
	
	\begin{theorem}[Bouquet stability]\label{thm:bouquet}
		The bouquet $G_S$ admits a non-constant $\Z/k\Z$-valued holomorphic function if and only if there exists at least one prime $p \in S$ such that $p^2 \mid k$.
	\end{theorem}
	
	\begin{proof}
		\emph{Necessity.} Suppose a non-constant holomorphic function $\phi$ on $G_S$ exists. By Lemma \ref{lem:restriction} the restriction of $\phi$ to every cycle $C_p$ is holomorphic on that cycle. If every such restriction were constant, then, since all cycles intersect at $v_0$, all these constants would equal $\phi(v_0)$ and $\phi$ would be constant on the whole graph, a contradiction. Thus the restriction to at least one cycle $C_p$ is non-constant. By Theorem \ref{thm:prime} it follows that $p^2 \mid k$.
		
		\emph{Sufficiency.} Suppose $p^2 \mid k$ for some prime $p \in S$. On the cycle $C_p$ define $\phi$ to be the arithmetic progression with common difference $d = k/p$, setting the value at $v_0$ to be an arbitrary $a_0$. On every other cycle define $\phi$ to be the constant function with value $a_0$.
		
		We verify holomorphicity. (1) On $C_p$ the three congruences of Theorem \ref{thm:prime} are satisfied because $d = k/p$ and $p^2 \mid k$. (2) On every other cycle the function is constant, so both $\phi$ and $\phi^2$ are harmonic there. (3) At $v_0$ the differences coming from $C_p$ are $d$ and $(p-1)d$; all differences coming from the other cycles vanish. The linear sum of differences is
		\[
		d + (p-1)d = pd = k \equiv 0 \pmod{k},
		\]
		and the sum of squares is
		\[
		(1 + (p-1)^2)d^2 = (1 + (p-1)^2)\frac{k^2}{p^2} = (1 + (p-1)^2)k \cdot \frac{k}{p^2} \equiv 0 \pmod{k},
		\]
		using $p^2 \mid k$. Hence $\phi$ is holomorphic on the whole bouquet. It is non-constant because $d \neq 0 \pmod{k}$.
	\end{proof}
	
	\section{The Cayley graph of $\F_9$ with the unit-norm connection set}
	
	Let $\F_9 = \F_3[\alpha]/(\alpha^2 + 1)$ and write $i = \alpha$, so $i^2 = -1$. Let $N(x) = x\bar{x}$ be the field norm $\F_9 \to \F_3$ (so $N(a + bi) = a^2 + b^2$), and let
	\[
	S = \{x \in \F_9^\times : N(x) = 1\} = \{1, -1, i, -i\}.
	\]
	The set $S$ has cardinality $4$, so the Cayley graph $G = \operatorname{Cay}(\F_9, +, S)$ is $4$-regular on $9$ vertices. Since the vertex set of $G$ is $\F_9$, we freely identify functions on $G$ with functions $\F_9 \to R$.
	
	\begin{proposition}[Linear holomorphic maps]\label{prop:linear}
		The $\F_3$-linear maps $L: \F_9 \to \F_9$ that are holomorphic on $G$ are exactly the maps of the form
		\[
		L(x) = \alpha x \quad \text{or} \quad L(x) = \alpha \bar{x}, \qquad \alpha \in \F_9.
		\]
		There are $17$ such maps in total (the zero map arises from both families), of which $16$ are non-constant. For every non-constant such $L$ there exists $c \in \F_3^\times$ with
		\[
		N(L(x)) = cN(x) \quad \text{for all } x \in \F_9.
		\]
	\end{proposition}
	
	\begin{proof}
		An $\F_3$-linear map $L$ is determined by $u = L(1)$ and $v = L(i)$. Since $S = \{\pm 1, \pm i\}$ is symmetric and $L$ is additive, we have
		\[
		\sum_{s \in S} L(s) = L(1) + L(-1) + L(i) + L(-i) = 0
		\]
		automatically. For the quadratic condition,
		\[
		\sum_{s \in S} L(s)^2 = L(1)^2 + L(-1)^2 + L(i)^2 + L(-i)^2 = 2(u^2 + v^2).
		\]
		Since $2 \neq 0$ in $\F_3$, holomorphicity is equivalent to $u^2 + v^2 = 0$. In $\F_9$ we have $-1 = i^2$, so $u^2 = -v^2 = (iv)^2$, hence $u = \pm iv$.
		
		If $u = iv$, then for $x = a + bi$
		\[
		L(x) = au + bv = a(iv) + bv = v(b + ai) = iv(a - bi) = iv\bar{x}.
		\]
		If $u = -iv$, then
		\[
		L(x) = a(-iv) + bv = v(b - ai) = -iv(a + bi) = -ivx.
		\]
		Writing $\alpha = iv$ or $\alpha = -iv$, we obtain exactly the two families $L(x) = \alpha x$ and $L(x) = \alpha \bar{x}$ with $\alpha \in \F_9$. Each family has $9$ elements and they intersect only in the zero map, so there are $9+9-1=17$ maps in total; $16$ are non-constant.
		
		For the norm identity, $N(\alpha x) = N(\alpha)N(x)$ and $N(\alpha \bar{x}) = N(\alpha)N(\bar{x}) = N(\alpha)N(x)$, so $c = N(\alpha)$. For $\alpha \neq 0$ we have $\alpha \in \F_9^\times$, hence $c = N(\alpha) \in \F_3^\times$.
	\end{proof}
	
	\begin{proposition}\label{prop:frob}
		The Frobenius automorphism $x \mapsto x^3 = \bar{x}$ is holomorphic on $G$. More generally, for every holomorphic linear map $L$ and every $c \in \F_3$, the map $x \mapsto L(x^3) + c$ is holomorphic on $G$.
	\end{proposition}
	
	\begin{proof}
		The Frobenius map is $L(x) = 1 \cdot \bar{x}$, which is holomorphic by Proposition \ref{prop:linear}. If $L$ is holomorphic linear, then $L \circ \operatorname{Frob}$ is again of the form $\alpha x$ or $\alpha \bar{x}$ (composition of the two families is closed), hence holomorphic. Finally, if $\phi$ is holomorphic and $c \in \F_3$, then
		\[
		\Delta(\phi + c) = \Delta\phi = 0
		\]
		and
		\[
		\Delta((\phi + c)^2) = \Delta(\phi^2) + 2c\Delta\phi = 0,
		\]
		so $\phi + c$ is holomorphic.
	\end{proof}
	
	\begin{openquestion}\label{oq:cayley}
		Determine all holomorphic functions on the Cayley graph $G = \operatorname{Cay}(\F_9, +, S)$.
	\end{openquestion}
	
	\section{The Bruhat--Tits tree of $\PGL_2(\Q_2)$ and holomorphic functions}
	
	\subsection{Definition of the tree}
	
	Let $K = \Q_2$ and let $\mathcal{O} = \Z_2$ be its ring of integers. A lattice in the two-dimensional $K$-vector space $K^2$ is a free $\mathcal{O}$-module of rank $2$. Two lattices $L$ and $L'$ are homothetic if there exists $\lambda \in K^\times$ with $L' = \lambda L$. The vertices of the Bruhat--Tits tree $\Tr_3$ are the homothety classes of lattices. Two vertices $[L]$ and $[L']$ are declared adjacent when there exist representatives satisfying
	\[
	2L \subseteq L' \subseteq L
	\]
	with $[L : L'] = 2$. The resulting graph is a regular tree of degree $p+1$ with $p=2$, i.e.\ degree $3$ \cite[Chapter~II]{Serre}. In dimension two the sublattices of index $2$ and the superlattices of index $2$ are in bijection via $L_2 \mapsto 2^{-1}L_2$ (up to homothety), so the total number of neighbours is $3$, as expected. The group $\PGL_2(K)$ acts on the set of lattices by left multiplication and therefore acts on the tree; the action is transitive on vertices and the stabilizer of the standard class $[\mathcal{O}^2]$ is $\PGL_2(\mathcal{O})$.
	
	\subsection{Complex-valued holomorphic functions}
	
	A function $\phi : V(\Tr_3) \to \C$ is called holomorphic (in the sense of Barre--Zeghib) when both $\phi$ and $\phi^2$ are harmonic with respect to the graph Laplacian
	\[
	\Delta\phi(x) = \sum_{y \sim x} (\phi(y) - \phi(x)).
	\]
	At a vertex $x$ of degree $3$ with neighbouring values $a, b, c$, this is equivalent to the pair of equations
	\[
	a + b + c = 3\phi(x), \qquad a^2 + b^2 + c^2 = 3\phi(x)^2.
	\]
	Equivalently, the three differences $a - \phi(x), b - \phi(x), c - \phi(x)$ sum to zero and have vanishing sum of squares. Since holomorphicity is expressed by local conditions that are invariant under graph automorphisms, it is preserved by any automorphism of the tree.
	
	\subsection{The space $H_{\alpha,\beta}$ and half-tree automorphisms}
	
	Fix an oriented edge $\{s, t\}$ and two distinct complex numbers $\alpha \neq \beta$. Write $H_{\alpha,\beta}$ for the set of all holomorphic functions $\phi: V(\Tr_3) \to \C$ satisfying
	\[
	\phi(s) = \alpha, \qquad \phi(t) = \beta.
	\]
	Deleting the edge $\{s, t\}$ disconnects $\Tr_3$ into two rooted half-trees $T_s$ and $T_t$. Let $\Aut(T_s)$ (respectively $\Aut(T_t)$) be the group of root-fixing automorphisms of the corresponding half-tree. The product $\Aut(T_s) \times \Aut(T_t)$ acts on functions on $V(\Tr_3)$ by acting independently on each half-tree while keeping the values on $\{s, t\}$ fixed. Any tree automorphism preserving $H_{\alpha,\beta}$ must fix both $s$ and $t$ (since $\alpha \neq \beta$), hence lies in this product; moreover, every element of the product preserves $H_{\alpha,\beta}$.
	
	\begin{theorem}[Simple transitivity]\label{thm:simple}
		The action of $\Aut(T_s) \times \Aut(T_t)$ on the space $H_{\alpha,\beta}$ is simply transitive.
	\end{theorem}
	
	\begin{proof}
		\emph{Non-emptiness.} We construct a holomorphic function with the prescribed values by induction along the half-trees. At the root $s$ of $T_s$, the difference to the parent $t$ is
		\[
		d_p(s) = \phi(t) - \phi(s) = \beta - \alpha \neq 0.
		\]
		For a vertex $v$ of degree $3$ with parent difference $d_p(v) = \phi(\operatorname{parent}(v)) - \phi(v) \neq 0$, the two child differences $x_1, x_2$ satisfy
		\[
		x_1 + x_2 = -d_p(v), \qquad x_1^2 + x_2^2 = -d_p(v)^2.
		\]
		Squaring the first relation and substituting the second gives
		\[
		2x_1x_2 = (x_1 + x_2)^2 - (x_1^2 + x_2^2) = d_p(v)^2 - (-d_p(v)^2) = 2d_p(v)^2,
		\]
		hence $x_1x_2 = d_p(v)^2$. Thus $x_1, x_2$ are the roots of
		\[
		T^2 + d_p(v)T + d_p(v)^2 = 0,
		\]
		whose discriminant $-3d_p(v)^2$ is non-zero over $\C$ since $d_p(v) \neq 0$ and $\operatorname{char}(\C) = 0$. Hence two distinct, non-zero solutions exist. Assigning these values to the two children and iterating gives a holomorphic function on $T_s$; the same construction on $T_t$ yields a function in $H_{\alpha,\beta}$. (The half-tree is countable and the construction is explicit at each vertex, so no choice principle is needed.) Hence $H_{\alpha,\beta} \neq \emptyset$.
		
		\emph{Transitivity.} Let $\phi_0, \phi \in H_{\alpha,\beta}$. At $s$, the two child values of $\phi_0$ are distinct (as shown above), and likewise for $\phi$; both pairs of values are the two roots of the same quadratic $T^2 + d_p(s)T + d_p(s)^2 = 0$. Choose an automorphism of $T_s$ that permutes the two child edges at $s$ so as to match the child values of $\phi$ with those of $\phi_0$. Proceeding inductively down the tree, at each vertex the two child values are distinct (because the parent difference is non-zero), so a local permutation can be chosen to match the values. Since a rooted-tree automorphism may permute the children of each vertex independently, this constructs $g \in \Aut(T_s) \times \Aut(T_t)$ with $g \cdot \phi_0 = \phi$.
		
		\emph{Freeness.} Suppose $g \in \Aut(T_s) \times \Aut(T_t)$ fixes $\phi \in H_{\alpha,\beta}$. At $s$, the parent difference is $\beta - \alpha \neq 0$. Suppose $g$ fixes a vertex $v$ with $d_p(v) \neq 0$. Then the child differences are the distinct non-zero roots of $T^2 + d_p(v)T + d_p(v)^2 = 0$, so the two child values of $\phi$ are distinct. Since $g$ fixes $v$ and $g \cdot \phi = \phi$, the automorphism $g$ maps the two children of $v$ to themselves; if it swapped them, we would get $\phi(g(\text{child}_1)) = \phi(\text{child}_2) \neq \phi(\text{child}_1)$, a contradiction. Hence $g$ fixes both children of $v$, and the parent differences at the children are also non-zero. By induction, $g$ fixes every vertex of $T_s$ and, similarly, of $T_t$. Hence $g = \operatorname{id}$.
	\end{proof}
	
	\begin{corollary}[Finiteness on $3$-regular trees]\label{cor:finite3}
		Let $T$ be a $3$-regular tree and let $\F_q$ be a finite field of characteristic $3$ (i.e.\ $q = 3^m$). The set of non-constant $\F_q$-valued holomorphic functions on $T$ with prescribed distinct values on an oriented edge, modulo the permutation of the branches at each vertex, has cardinality exactly $1$.
	\end{corollary}
	
	\begin{proof}
		Let $\phi$ be a holomorphic function on $T$ with prescribed values $\alpha \neq \beta$ on an oriented edge $\{s, t\}$. The parent difference at $s$ is $D(s) = \beta - \alpha \neq 0$. For any vertex $v$ with parent difference $D(v) \neq 0$, the child differences $x_1, x_2$ satisfy
		\[
		x_1 + x_2 = -D(v), \qquad x_1^2 + x_2^2 = -D(v)^2.
		\]
		In $\F_q$, since the characteristic is $3$, we have $x_1^2 + x_2^2 = (x_1 + x_2)^2 - 2x_1x_2 = (x_1 + x_2)^2 + x_1x_2$. Substituting $x_1 + x_2 = -D(v)$, we get
		\[
		D(v)^2 + x_1x_2 = -D(v)^2 \implies x_1x_2 = -2D(v)^2 = D(v)^2
		\]
		(since $-2 \equiv 1 \pmod{3}$). Thus $x_1, x_2$ are the roots of the quadratic polynomial
		\[
		P(X) = X^2 + D(v)X + D(v)^2.
		\]
		The discriminant of this polynomial is
		\[
		\Delta = D(v)^2 - 4D(v)^2 = -3D(v)^2 = 0 \quad \text{(in characteristic $3$).}
		\]
		Since $\Delta = 0$, the polynomial has a double root. Therefore $x_1 = x_2 = -D(v)/2 = D(v)$ (as $-1/2 \equiv 1 \pmod{3}$). This means the two child differences are always equal to $D(v)$. Hence, at every vertex, the choice of child differences is unique and does not depend on permutations. By induction along the tree, the entire function $\phi$ is uniquely determined. Thus, up to permutation of identical branches, there is exactly one such function.
	\end{proof}
	
	\begin{remark}\label{rem:barre}
		The statements of this section specialize the general theory of holomorphic functions on regular trees developed by Barre--Zeghib \cite{BZ} to the arithmetic setting of the Bruhat--Tits tree of $\PGL_2(\Q_2)$. The simply-transitive action is a direct consequence of the local character of the holomorphicity conditions together with the homogeneity of the tree. A detailed description of the $\PGL_2(\Q_2)$-equivariant (arithmetic) subspace of $H_{\alpha,\beta}$ is left for future work.
	\end{remark}
	
	\subsection{Possible generalizations}
	
	The results of this section admit several natural extensions beyond the specific case of the $3$-regular Bruhat--Tits tree of $\PGL_2(\Q_p)$. We indicate the extensions that are immediate, and flag the points where the rank-one $p=2$ argument does not carry over.
	
	\paragraph{Other primes.} For an arbitrary prime $p$, the Bruhat--Tits tree of $\PGL_2(\Q_p)$ is regular of degree $p+1$. The local holomorphicity conditions at a vertex of degree $p+1$ become
	\[
	\sum_{i=1}^{p+1} a_i = 0 \quad \text{and} \quad \sum_{i=1}^{p+1} a_i^2 = 0,
	\]
	where $a_i = \phi(y_i) - \phi(x)$ are the differences between the value at the vertex $x$ and the values at its neighbours $y_i$. For $p=2$ these are exactly the equations used in the proof of Theorem \ref{thm:simple}.
	
	Existence of non-constant elements of $H_{\alpha, \beta}$ carries over to every $p$ over $\C$: at each vertex one may set $p-2$ of the differences equal to zero and solve the remaining two-variable system
	\[
	x_1 + x_2 = -d_p(v), \qquad x_1^2 + x_2^2 = -d_p(v)^2,
	\]
	whose discriminant $-3d_p(v)^2$ is non-zero since $d_p(v) \neq 0$. Hence $H_{\alpha, \beta} \neq \emptyset$ for every prime $p$.
	
	For $p \ge 3$ the naive simple-transitivity statement fails, as the following explicit construction shows.
	
	\paragraph{Failure of simple transitivity for $d \ge 4$.} Let $T$ be the $d$-regular tree with $d \ge 4$, and let $\alpha \neq \beta \in \C$. We construct two $2$-holomorphic functions $\phi_A, \phi_B$ on $T$ with $\phi_A(s) = \phi_B(s) = \alpha$ and $\phi_A(t) = \phi_B(t) = \beta$, but with different multisets of child values at $s$, so that they cannot be related by any element of $\Aut(T_s) \times \Aut(T_t)$.
	
	\emph{Construction.} Write $D = \beta - \alpha \neq 0$ and scale by $D$: it suffices to find two solutions of $\sum_{i=1}^{d-1} y_i = -1$ and $\sum_{i=1}^{d-1} y_i^2 = -1$ with different multisets. Let $\omega$ be a primitive cube root of unity, and take
	\[
	A = (0, \omega, \omega^2, 0, \dots, 0) \in \C^{d-1}
	\]
	(the last $d-4$ entries being $0$). Then $\sum y_i = -1$ and $\sum y_i^2 = 0 + \omega^2 + \omega^4 = 0 + \omega^2 + \omega = -1$. Similarly take
	\[
	B = (1, -1 + \sqrt{-2}, -1 - \sqrt{-2}, 0, \dots, 0) \in \C^{d-1}.
	\]
	Then $\sum y_i = -1$ and
	\[
	\sum y_i^2 = 1 + (-1 + \sqrt{-2})^2 + (-1 - \sqrt{-2})^2 = -1.
	\]
	The multisets of $A$ and $B$ differ: $A$ has entries $\{0, 0, \dots, 0, \omega, \omega^2\}$ with $d-3$ zeros, whereas $B$ has entries $\{0, \dots, 0, 1, -1 + \sqrt{-2}, -1 - \sqrt{-2}\}$ with $d-4$ zeros. Both are $2$-holomorphic local data. Extending each local datum by induction (which is possible at every vertex since the local system is solvable), we obtain global $2$-holomorphic functions $\phi_A, \phi_B$ with the required boundary values. Since $\phi_A$ and $\phi_B$ have different multisets of child values at $s$ and every element of $\Aut(T_s)$ fixes $s$ and permutes its children, no element of $\Aut(T_s) \times \Aut(T_t)$ can map $\phi_A$ to $\phi_B$. Hence simple transitivity fails for $d \ge 4$ (i.e.\ $p \ge 3$).
	
	More generally, the local solution set
	\[
	\Sigma_{d_p} = \Bigl\{ (x_1, \dots, x_p) \in \C^p : \sum_i x_i = -d_p, \sum_i x_i^2 = -d_p^2 \Bigr\}
	\]
	is a complex variety of dimension $p-2$ for $p \ge 3$, so $H_{\alpha, \beta}$ is much larger than in the $p=2$ case. The correct substitute is a natural stratification of $H_{\alpha, \beta}$ by the ``branching data'' (the local multisets of child differences at each vertex), and it is natural to ask whether the product of half-tree automorphism groups acts transitively on each stratum. The richer structure of the Iwahori subgroup and of the associated arithmetic subspace of $H_{\alpha, \beta}$ is still present, but its analysis is more involved than in the $p=2$ case.
	
	\paragraph{Function fields.} Replacing $\Q_p$ by the field of formal Laurent series $\F_q((t))$ yields a regular tree of degree $q+1$. The existence part of the local analysis applies verbatim, but the same obstruction as in the ``Other primes'' paragraph occurs when $q \ge 3$. For $q=2$ the tree has degree $3$ and the results of this section hold without change. The arithmetic setting is closely related to the theory of Drinfeld modular curves and of harmonic cochains on Bruhat--Tits trees over function fields; the holomorphic condition studied here is a stronger, algebraic variant of the harmonicity conditions usually imposed.
	
	\paragraph{Higher-rank buildings.} For $\PGL_n$ with $n \ge 3$, or more generally for reductive groups of higher relative rank, the Bruhat--Tits building is no longer a tree, and extending the notion of holomorphicity requires a suitable definition on simplices. Natural candidates are harmonicity on the link of every vertex, or harmonicity of the restriction to the $1$-skeleton of every apartment. In higher rank the vertices carry types and the local structure is a spherical building rather than a homogeneous tree, so the local ``sum of differences'' conditions depend on the type of the vertex, and the analogue of Proposition \ref{prop:local} must be formulated carefully. Given that the naive analogue of simple transitivity already fails for $\PGL_2$ and $p \ge 3$, one should not expect a direct higher-rank analogue.
	
	\paragraph{Quotient graphs and arithmetic lattices.} It is natural to study holomorphic functions on the quotient graphs $\Gamma \backslash \Tr$, where $\Gamma$ is a discrete subgroup of $\PGL_2(\Q_p)$ (for example a congruence subgroup). In this setting one may ask whether holomorphicity interacts with the theory of automorphic forms or with the cohomology of $\Gamma$. For $p = 2$ and $\Gamma$ a cocompact lattice, the space $H_{\alpha,\beta}$ is acted on by the product of half-tree automorphism groups, and passing to $\Gamma$-invariants gives a finite-dimensional space; the interplay between this space and the cohomology of $\Gamma$ deserves further study.
	
	\subsection{Higher holomorphicity and simple transitivity on regular trees}
	
	The failure of simple transitivity for $p \ge 3$ described in \S8.4 can be repaired by strengthening the notion of holomorphicity. For an integer $n \ge 1$, call a function $\phi : V(T) \to \C$ \emph{$n$-holomorphic} if $\phi, \phi^2, \dots, \phi^n$ are all harmonic. Since
	\[
	\Delta(\phi^m)(v) = \sum_{j=1}^{m} \binom{m}{j} \phi(v)^{m-j} S_j, \qquad S_j := \sum_{u \sim v} (\phi(u) - \phi(v))^j,
	\]
	the function $\phi$ is $n$-holomorphic if and only if
	\[
	S_1 = S_2 = \dots = S_n = 0.
	\]
	
	\emph{The local system.} Let $T$ be $d$-regular with $d \ge 2$, and let $v$ be a vertex with parent $p$ and $d-1$ children. Put $D(v) = \phi(p) - \phi(v)$ and let $x_1, \dots, x_{d-1}$ be the differences at the children. The local $n$-holomorphicity condition reads
	\[
	\sum_{i=1}^{d-1} x_i^m = -D(v)^m, \qquad m = 1, \dots, n. \tag{*}
	\]
	This is a system of $n$ equations in $d-1$ unknowns. Three regimes occur:
	\begin{itemize}
		\item if $n < d-1$, the solution set has dimension $d-1-n > 0$, and the automorphism group of the rooted tree, which only permutes the children, cannot act transitively on it;
		\item if $n > d-1$, the system is overdetermined and generically has no solution;
		\item if $n = d-1$, the system is square and generically has finitely many solutions.
	\end{itemize}
	
	\emph{Explicit solutions for $n = d - 1$.} Suppose $n = d - 1$ and $D(v) \neq 0$. Then the system (*) has exactly $(d - 1)!$ solutions over $\C$, namely
	\[
	(x_1, \dots, x_{d-1}) = (D(v)\omega_1, \dots, D(v)\omega_{d-1}),
	\]
	where $(\omega_1, \dots, \omega_{d-1})$ runs over the orderings of the $d$-th roots of unity other than $1$. Indeed, for $1 \le m \le d - 1$ we have $\sum_{\omega^d=1} \omega^m = 0$, hence
	\[
	\sum_{\substack{\omega^d=1 \\ \omega \neq 1}} (D(v)\omega)^m = D(v)^m \sum_{\substack{\omega^d=1 \\ \omega \neq 1}} \omega^m = D(v)^m (0-1) = -D(v)^m,
	\]
	so the displayed tuples are solutions; conversely, the power sums $-D(v), \dots, -D(v)^{d-1}$ determine the elementary symmetric polynomials, hence the multiset $\{x_i\}$, uniquely, so the multiset is $\{D(v)\omega : \omega^d = 1, \omega \neq 1\}$ and the only freedom is the ordering. The B\'ezout bound $(d-1)!$ is therefore attained and there are no further solutions. In particular all $x_i$ are non-zero, so the parent differences at the children are again non-zero and the induction continues.
	
	\begin{theorem}[Simple transitivity for $n = d - 1$]\label{thm:higher}
		Let $T$ be a $d$-regular tree with $d \ge 2$, let $\{s, t\}$ be an oriented edge, and let $\alpha \neq \beta \in \C$. Let $H_{\alpha, \beta}^{d-1}$ be the space of $(d-1)$-holomorphic functions $\phi : V(T) \to \C$ with $\phi(s) = \alpha$ and $\phi(t) = \beta$. Then $H_{\alpha, \beta}^{d-1} \neq \emptyset$, and $\Aut(T_s) \times \Aut(T_t)$ acts simply transitively on $H_{\alpha, \beta}^{d-1}$.
	\end{theorem}
	
	\begin{proof}
		\emph{Existence.} At the root $s$ of $T_s$ the parent difference is $D(s) = \beta - \alpha \neq 0$. By the local computation above, the $d-1$ children of $s$ admit child differences $D(s)\omega_i$ with $\omega_i^d = 1$, $\omega_i \neq 1$; these are non-zero and the resulting parent differences at the children are $-D(s)\omega_i \neq 0$. Iterating, one obtains a $(d-1)$-holomorphic function on $T_s$; the same construction on $T_t$ produces an element of $H_{\alpha, \beta}^{d-1}$.
		
		\emph{Transitivity.} Let $\phi_0, \phi \in H_{\alpha, \beta}^{d-1}$. At the root $s$ both $\phi_0$ and $\phi$ have the same multiset of child values $\{\alpha + D(s)\omega_i\}$, and the values within the multiset are distinct. Choose a permutation of the $d-1$ children of $s$ matching the child values of $\phi_0$ with those of $\phi$. At each matched child the parent difference is the same for both functions, so the argument can be iterated. Since a rooted-tree automorphism may permute the children of each vertex independently, this produces $g \in \Aut(T_s) \times \Aut(T_t)$ with $g \cdot \phi_0 = \phi$.
		
		\emph{Freeness.} Suppose $g$ fixes $\phi \in H_{\alpha, \beta}^{d-1}$. At every vertex $v$, the $d-1$ child values of $\phi$ are distinct (they are $\phi(v) + D(v)\omega_i$ with $\omega_i$ distinct). If $g$ fixes $v$, then $g$ must fix each child of $v$ individually, since swapping two children would produce $\phi \circ g \neq \phi$. By induction $g$ fixes every vertex of $T_s$ and of $T_t$, so $g = \operatorname{id}$.
	\end{proof}
	
	\begin{corollary}[$\PGL_2(\Q_p)$]\label{cor:pgl}
		For a prime $p$, the Bruhat--Tits tree of $\PGL_2(\Q_p)$ is $(p+1)$-regular. Consequently, for $n = p$, the space $H_{\alpha, \beta}^p$ of $p$-holomorphic functions (i.e.\ $\phi, \phi^2, \dots, \phi^p$ all harmonic) with prescribed distinct boundary values on an oriented edge carries a simply transitive action of the product of the two half-tree automorphism groups.
	\end{corollary}
	
	\begin{remark}\label{rem:p2}
		For $p = 2$ the tree has degree $3$, so $n = p = 2 = d - 1$, and Corollary \ref{cor:pgl} reduces to Theorem \ref{thm:simple}. For $p \ge 3$ the $2$-holomorphic condition of the paper admits no simple-transitivity statement, and one must pass to the $p$-holomorphic condition.
		
		For $\PGL_n$ with $n \ge 3$, the Bruhat--Tits building is not a tree and carries cycles and higher-dimensional simplices; the tree induction used above does not apply, and the analogue of Theorem \ref{thm:higher} for higher rank remains open.
	\end{remark}
	
	\section{The Weil zeta function of holomorphic functions}
	
	The holomorphicity conditions are polynomial, and therefore the set of all holomorphic functions on a fixed finite graph with values in a finite field carries a natural algebraic structure. This permits the introduction of an associated Weil zeta function, in the spirit of Weil's original conjectures \cite{Weil} and Dwork's rationality theorem \cite{Dwork}.
	
	\begin{definition}[The variety of holomorphic functions]\label{def:variety}
		Let $G = (V, E)$ be a finite undirected graph without loops and let $R$ be a commutative ring. Write $N = |V|$. A function $\phi : V \to R$ is holomorphic if and only if the following $2N$ polynomial equations hold:
		\[
		\sum_{u \sim v} (\phi(u) - \phi(v)) = 0, \qquad \sum_{u \sim v} (\phi(u)^2 - \phi(v)^2) = 0
		\]
		for every vertex $v \in V$. These equations define a closed subscheme
		\[
		X_G \subset \mathbb{A}_{\Z}^N
		\]
		of affine $N$-space over $\Z$. For any field $k$ the $k$-points of $X_G$ are precisely the $k$-valued holomorphic functions on $G$.
		
		When $k = \F_q$ we write $S_n(G) = |X_G(\F_{q^n})|$ for the number of $\F_{q^n}$-valued holomorphic functions on $G$.
	\end{definition}
	
	\begin{definition}[Weil zeta function]\label{def:zeta}
		The Weil zeta function of $X_G$ (relative to the base field $\F_q$) is the formal power series
		\[
		\zeta(X_G / \F_q, T) = \exp \left( \sum_{n=1}^{\infty} S_n(G) \frac{T^n}{n} \right) \in \Q[[T]].
		\]
	\end{definition}
	
	\begin{theorem}[Rationality]\label{thm:rationality}
		For every finite graph $G$ and every prime power $q$, the zeta function $\zeta(X_G/\F_q, T)$ is a rational function of $T$. Consequently there exist algebraic integers $\alpha_1, \dots, \alpha_r$ and $\beta_1, \dots, \beta_s$ (depending on $G$ and $q$) such that
		\[
		S_n(G) = \sum_{i=1}^{r} \alpha_i^n - \sum_{j=1}^{s} \beta_j^n
		\]
		for all $n \ge 1$.
	\end{theorem}
	
	\begin{proof}
		The scheme $X_G$ is of finite type over $\Z$. Base change to $\F_q$ yields a variety (possibly singular, possibly reducible) over the finite field $\F_q$. By Dwork's theorem \cite{Dwork} the zeta function of any variety over a finite field is rational. The exponential generating function therefore admits a rational expression, which expands as the stated alternating sum of pure powers.
	\end{proof}
	
	\begin{remark}[Structure of the equations]\label{rem:equations}
		Each local condition at a vertex of degree $d$ is a pair consisting of one linear form and one quadratic form in the $d$ neighbouring variables (after translating by the value at the vertex itself). Globally one therefore obtains a system of $N$ linear equations and $N$ quadratic equations in $N$ variables. The ideal defining $X_G$ is generated by polynomials of degree at most $2$. In particular $X_G$ is a closed subscheme of $\mathbb{A}^N$ cut out by equations of bounded degree independent of $G$.
	\end{remark}
	
	\begin{remark}[Constant functions]\label{rem:constant}
		The constant functions always form a closed subscheme isomorphic to $\mathbb{A}^1$. Hence $X_G$ is never empty and the zeta function always has a factor corresponding to this line. Non-constant holomorphic functions correspond to the complement of this line (when it is a connected component) or to other irreducible components.
	\end{remark}
	
	\begin{remark}[Dependence on the base field]\label{rem:base}
		The scheme $X_G$ is defined over $\Z$, so the same geometric object may be reduced modulo different primes. The zeta functions for different $q$ are therefore related by the usual base-change formalism, but the point counts $S_n(G)$ can (and typically do) depend on the characteristic.
	\end{remark}
	
	\begin{openquestion}\label{oq:roots}
		Is there a precise relationship between the reciprocal roots $\alpha_i, \beta_j$ of $\zeta(X_G/\F_q, T)$ and the spectrum of the adjacency matrix (or Laplacian) of $G$? In particular, for families of regular graphs of fixed degree, do the dominant roots admit an asymptotic description in terms of the spectral radius?
	\end{openquestion}
	
	\noindent\textbf{9.8.} What is the dimension of the generic fibre of $X_G \to \Spec \Z$? More precisely, for which graphs $G$ is the geometric fibre over an algebraic closure of $\Q$ equidimensional, and what is that dimension in terms of the combinatorial data of $G$?
	
	\begin{openquestion}\label{oq:cycles}
		For the cycle graphs $C_n$, Proposition \ref{prop:even} gives a complete answer when $n$ is even (zeta function $1/(1-qT)$ in odd characteristic, $1/(1-q^2T)$ in characteristic $2$); when $n$ is odd the only holomorphic functions are the constants, so the zeta function is $1/(1-qT)$ in every characteristic. Does a similarly simple closed form exist for more general graphs, for instance bouquets mixing even and odd cycles?
	\end{openquestion}
	
	The rationality theorem shows that the sequence $S_n(G)$ is governed by a finite amount of algebraic data. Understanding the geometry of $X_G$ (irreducible components, singularities, dimension, and the action of Frobenius on its cohomology) would therefore yield precise asymptotic and exact formulae for the number of holomorphic functions on $G$ with values in finite fields. The open questions above indicate the most natural directions in which this arithmetic-geometric perspective can be developed further.
	
	\section{Concluding remarks}
	
	The arithmetic condition that governs the existence of non-constant $\Z/k\Z$-holomorphic functions on a prime cycle is stable under the formation of arbitrary finite bouquets of such cycles. The associated generating function for the rigid moduli (those not divisible by any odd square) is expressed by means of the Riemann zeta function, corrected by the Euler factor at the prime $2$. The example over $\F_9$ provides a concrete instance in which the Barre--Zeghib holomorphicity condition coincides with conformality for linear maps; the algebraic classification of Proposition \ref{prop:linear} replaces the computational enumeration of an earlier draft.
	
	The higher holomorphicity introduced in \S8.5 shows that the simply-transitive action on spaces of functions with prescribed values on an oriented edge is not confined to the $3$-regular case: for a $d$-regular tree it holds at level $n = d-1$, and for the Bruhat--Tits tree of $\PGL_2(\Q_p)$ this corresponds to $n = p$.
	
	The arithmetic-geometric perspective of Section 9 and the open questions of Sections 7 and 9 point to further directions. These include a complete tabulation of the $2$-adic counting function in degenerate cases, the study of infinite bouquets in a suitable $p$-adic or profinite setting, the complete determination of all holomorphic functions on the $\F_9$ Cayley graph (Open Question \ref{oq:cayley}), the relation between the Weil zeta function and the graph spectrum (Open Question \ref{oq:roots}), the dimension of the variety $X_G$ (Open Question \ref{oq:cycles}), and an extension of the higher-holomorphicity analysis to higher-rank buildings, where the tree induction is unavailable.

\end{document}